\documentclass[11pt]{amsart}

\usepackage{amsmath}
\usepackage{amssymb}
\usepackage{amscd}
\usepackage{mathrsfs}

\usepackage{xypic}
\xyoption{all}

\newcommand{\nc}{\newcommand}

\nc{\lan}{\big\langle}
\nc{\ran}{\big\rangle}

\nc{\kk}{{\mathbb{C}}}

\nc{\LL}{{\mathbb{L}}}
\nc{\PP}{{\mathbb{P}}}
\nc{\QQ}{{\mathbb{Q}}}
\nc{\RR}{{\mathbb{R}}}
\nc{\TT}{{\mathbb{T}}}
\nc{\ZZ}{{\mathbb{Z}}}

\nc{\CA}{{\mathcal{A}}}
\nc{\CB}{{\mathcal{B}}}
\nc{\D}{{\mathcal{D}}}
\nc{\CE}{{\mathcal{E}}}
\nc{\CI}{{\mathcal{I}}}
\nc{\CK}{{\mathcal{K}}}
\nc{\CL}{{\mathcal{L}}}
\nc{\CM}{{\mathcal{M}}}
\nc{\CN}{{\mathcal{N}}}
\nc{\CO}{{\mathcal{O}}}
\nc{\CS}{{\mathcal{S}}}
\nc{\CT}{{\mathcal{T}}}
\nc{\CU}{{\mathcal{U}}}

\nc{\BB}{{\mathbf{B}}}
\nc{\BD}{{\mathbf{D}}}
\nc{\BG}{{\mathbf{G}}}
\nc{\BO}{{\mathbf{O}}}
\nc{\BP}{{\mathbf{P}}}
\nc{\BT}{{\mathbf{T}}}
\nc{\bm}{{\mathbf{m}}}
\nc{\bu}{{\mathbf{u}}}
\nc{\bx}{{\mathbf{x}}}

\nc{\SJ}{{\mathsf{J}}}

\nc{\TBP}{{\tilde{\BP}}}

\nc{\TD}{{\widetilde{\D}}}
\nc{\TCA}{{\tilde{\CA}}}
\nc{\TY}{{\widetilde{Y}}}

\nc{\RHom}{\mathop{\mathsf{RHom}}\nolimits}
\nc{\Hom}{\mathop{\mathsf{Hom}}\nolimits}
\nc{\Ext}{\mathop{\mathsf{Ext}}\nolimits}
\nc{\Tor}{\mathop{\mathsf{Tor}}\nolimits}
\nc{\Pic}{\mathop{\mathsf{Pic}}\nolimits}
\renewcommand{\Im}{\mathop{\mathsf{Im}}\nolimits}
\nc{\Br}{\mathop{\mathsf{Br}}\nolimits}
\nc{\Cone}{\mathop{\mathsf{Cone}}\nolimits}
\nc{\codim}{\mathop{\mathsf{codim}}\nolimits}
\nc{\sing}{{\mathsf{sing}}}
\nc{\perf}{{\mathsf{perf}}}
\nc{\supp}{{\mathsf{supp}}}
\nc{\Ker}{{\mathsf{Ker}}}
\nc{\Coker}{{\mathsf{Coker}}}
\nc{\End}{{\mathsf{End}}}
\nc{\Bl}{{\mathsf{Bl}}}
\nc{\Tr}{\mathop{\mathsf{Tr}}}
\nc{\Pf}{{\mathsf{Pf}}}
\nc{\Gr}{{\mathsf{Gr}}}
\nc{\SGr}{{\mathsf{SGr}}}
\nc{\LGr}{{\mathsf{LieGr}}}
\nc{\GTGr}{{\BG_2\mathsf{Gr}}}
\nc{\OGr}{{\mathsf{OGr}}}
\nc{\OFl}{{\mathsf{OFl}}}
\nc{\Fl}{{\mathsf{Fl}}}
\nc{\Spin}{{\mathsf{Spin}}}
\nc{\GL}{{\mathsf{GL}}}
\nc{\SL}{{\mathsf{SL}}}
\nc{\fg}{{\mathfrak{g}}}
\nc{\fsl}{{\mathfrak{sl}}}
\nc{\fso}{{\mathfrak{so}}}
\nc{\PGL}{{\mathsf{PGL}}}
\nc{\Cliff}{{\mathscr{C}\!{\ell}}}
\nc{\ch}{{\mathsf{ch}}}
\nc{\td}{{\mathsf{td}}}
\nc{\id}{{\mathsf{id}}}

\nc{\TGr}{\widehat\Gr}
\nc{\TX}{\widehat{X}}
\nc{\HPB}{{\widehat{\PP(B)}}}

\theoremstyle{plain}

\newtheorem{theorem}{Theorem}[section]

\newtheorem{question}[theorem]{Question}
\newtheorem{lemma}[theorem]{Lemma}
\newtheorem{proposition}[theorem]{Proposition}
\newtheorem{corollary}[theorem]{Corollary}

\theoremstyle{definition}

\theoremstyle{remark}

\newtheorem{remark}[theorem]{Remark}

\title{On K\"uchle manifolds with Picard number greater than 1}
\author{Alexander Kuznetsov}
\address{\sloppy
\parbox{0.9\textwidth}{
Algebra Section, Steklov Mathematical Institute, 8 Gubkin str., Moscow 119991 Russia
\hfill\\[5pt]
The Poncelet Laboratory, Independent University of Moscow\hfill\\[5pt]
Laboratory of Algebraic Geometry, SU-HSE\hfill
}\bigskip}
 \email{{\tt  akuznet@mi.ras.ru}}
\date{}
\thanks{This work is supported by the RSF under a grant 14-50-00005}

\begin{document}

\begin{abstract}
We describe the geometry of K\"uchle varieties (i.e.\ Fano 4-folds of index 1 contained in the Grassmannians as zero loci of equivariant vector bundles) 
with Picard number greater than 1 and the structure of their derived categories.
\end{abstract}

\maketitle

\section{Introduction}

In 1995 Oliver K\"uchle classified in~\cite{kuchle1995fano} all Fano manifolds of degree 4 and index 1 which can be represented
as zero loci of equivariant vector bundles on Grassmannians. His list consisting of 21 examples
can be considered as a first step towards the classification of Fano 4-folds and is still the main 
source of examples of those. However, the actual geometry of K\"uchle varieties was not investigated
so far. This paper is a first step in this direction.

Here we consider only those K\"uchle varieties which have Picard number greater than 1.
There are four such examples in the list:
\begin{description}
\item[(b4)] the zero locus of a global section of the vector bundle $S^2\CU^\vee \oplus \CO(2)$ on $\Gr(2,6)$;
\item[(b9)] the zero locus of a global section of the vector bundle $S^2\CU^\vee \oplus S^2\CU^\vee$ on $\Gr(2,7)$;
\item[(c7)] the zero locus of a global section of the vector bundle $\Lambda^2\CU^\perp(1) \oplus \CO(1)$ on $\Gr(3,8)$; and
\item[(d3)] the zero locus of a global section of the vector bundle $\Lambda^2\CU^\vee \oplus \Lambda^2\CU^\vee \oplus \CO(1)$ on $\Gr(5,10)$.
\end{description}
Here $\CU$ and $\CU^\perp$ denote the tautological vector subbundles of ranks $k$ and $n-k$ on the Grassmannian $\Gr(k,n)$ respectively,
while $\CO(1)$ stands for the ample generator of its Picard group. The main result of the paper is an alternative description
of these varieties, which gives a much better understanding of their geometry. In particular, we apply this to describe the structure
of the derived categories of coherent sheaves of these varieties.

In fact, for variety of type (b4) an alternative description is evident, and in case (b9) it is provided by a recent result of 
Casagrande~\cite{casagrande2014rank}. Furthermore, the case (d3) is quite simple and I would guess that it should be known to experts, 
although I couldn't find a reference. Finally, the case (c7) is not that simple but still is quite manageable. 
The description of this variety is the main result of the paper.

Of course, a simplified description which we get is possible for the K\"uchle varieties under consideration because their Picard group
has big rank and so they have additional structures which can be used. It is hard to expect something similar
for other K\"uchle varieties. Their geometry should be much more complicated, but at the same time much more interesting.

I would like to point out two questions about the geometry of other K\"uchle varieties, which seem to me to be of special interest. 
First, in the K\"uchle's list there are two varieties which have the same collections of discrete invariants:
\begin{description}
\item[(b3)] the zero locus of a global section of the vector bundle $\Lambda^3\CU^\perp(2)$ on $\Gr(2,6)$, and
\item[(b7)] the zero locus of a global section of the vector bundle $\CO(1)^{\oplus 6}$ on $\Gr(2,7)$.
\end{description}

\begin{question}\label{question-b3-b7}
Are K\"uchle varieties of types {\rm{(b3)}} and {\rm{(b7)}} deformation equivalent?
\end{question}

\begin{remark}
Soon after a preliminary version of this paper appeared, Laurent Manivel gave in~\cite{manivel15} an affirmative answer to this question.
\end{remark}


The second question concerns yet another variety:
\begin{description}
\item[(c5)] the zero locus of $\Lambda^2\CU^\vee \oplus \CU^\perp(1) \oplus \CO(1)$ on $\Gr(3,7)$.
\end{description}
The computations in K\"uchle's paper show that its Hodge diamond has the Hodge diamond of a K3 surface sitting inside,
so one can expect this variety to be analogous to a cubic 4-fold. In particular, it is natural to expect that its derived category
has a noncommutative K3 surface as one of the components (cf.~\cite{kuznetsov2010derived}) and that the Hilbert schemes 
of rational curves on it give rise to some hyper-K\"ahler varieties, like the Fano scheme of lines (\cite{beauville1985variete}), 
or the Hilbert scheme of twisted cubics (\cite{lehn2013twisted}) on a cubic 4-fold. Also the rationality questions for these
varieties might be interesting.

\begin{question}\label{question-c5}
Is there a semirthogonal decomposition of the derived category of a K\"uchle variety of type {\rm(c5)}
with a noncommutative {\rm K3} surface as one of the components? Does it give a hyper-K\"ahler structure
on some Hilbert schemes of curves. What can be said about rationality of varieties of that type?
\end{question}

In fact, varieties of type (d3) and (c7) discussed in this paper also have the Hodge diamond of a K3 surface inside,
and we could also ask Question~\ref{question-c5} for those. However, the simplified description of these
varieties obtained in sections~\ref{section-d3} and~\ref{section-c7} answers (is a sense) this question. 
For varieties of type (d3) the answer is trivial
(all varieties are rational, their derived category contains a commutative K3 surface as a component,
and the Hilbert schemes should reduce to moduli spaces of sheaves on the associated K3 surface).
For varieties (c7) the question 
reduces to the same question for special cubic 4-folds.


\medskip

{\bf Acknowledgements:} I am very grateful to Atanas Iliev who informed me about the K\"uchle's paper
and attracted my attention to the question. I would also like to thank Cinzia Casagrande and Sergei Galkin
for discussions of the geometry of the variety of type (b9).

\section{Varieties of types (b4) and (b9)}\label{section-b4-b9}

\subsection{Type (b4)}

Recall that a K\"uchle variety of type (b4) is the zero locus of a global section of the vector bundle $S^2\CU^\vee \oplus \CO(2)$ on $\Gr(2,6)$.

\begin{proposition}
A smooth K\"uchle variety $X_{b4}$ of type {\rm(b4)} is an intersection of divisors of bidegree $(1,1)$ and $(2,2)$ in $\PP^3\times\PP^3$.
\end{proposition}
\begin{proof}
A global section of the vector bundle $S^2\CU^\vee$ on $\Gr(2,6)$ is given by a quadratic form on the underlying
6-dimensional vector space $V = \kk^6$ and its zero locus is nothing but the Fano scheme of lines on the associated quadric in $\PP(V)$.
If the quadric were degenerate the zero locus would have at least 3-dimensional singularity and hence $X_{b4}$ would also be singular.
So we can assume that the quadric is nondegenerate.
But a nondegenerate quadric in $\PP(V) = \PP^5$ can be identified with the Grassmannian $\Gr(2,W) \subset \PP(\Lambda^2W)$
for a 4-dimensional vector space $W = \kk^4$, and the Fano scheme of lines on $\Gr(2,W)$ then identifies with the flag variety $\Fl(1,3;W)$.
Further, the flag variety $\Fl(1,3;W)$ is the divisor of bidegree $(1,1)$ in $\PP(W)\times\PP(W^\vee)$ (corresponding to the natural
pairing between $W$ and $W^\vee$), and under this identification the restriction of the generator of $\Pic(\Gr(2,V))$ corresponds
to the class of the divisors of bidegree $(1,1)$ in $\Fl(1,3;W) \subset \PP(W)\times\PP(W^\vee)$. Consequently, taking an additional 
zero locus of a section of the line bundle $\CO(2)$ on $\Gr(2,V)$ is equivalent to taking the zero locus of a section 
of the line bundle $\CO(2,2)$ on $\Fl(1,3;W)$.
\end{proof}

Denote by $h$ the pullback via the projection $X_{b4} \to \PP(W)$ of the positive generator 
of $\Pic(\PP(W))$.

\begin{corollary}
A K\"uchle variety $X_{b4}$ of type {\rm(b4)} has a structure of a conic bundle over $\PP^3$. For generic
$X_{b4}$ the discriminant of the conic bundle is an octic surface in $\PP^3$ with $80$ ordinary double points.
\end{corollary}
\begin{proof}
The projection $\Fl(1,3;W) \to \PP(W)$ is 
the projectivization 
of the bundle $\Omega_{\PP(W)}$ and a divisor of bidegree $(2,2)$ gives a quadratic form on $\Omega_{\PP(W)}$, i.e.\ 
a symmetric morphism $\Omega_{\PP(W)} \to T_{\PP(W)}$. By~\cite{harris1984symmetric} the classes of the discriminant and the corank 2 locus of this map
can be computed as $2c_1(T_{\PP(W)})$ and $4(c_1(T_{\PP(W)})c_2(T_{\PP(W)}) - c_3(T_{\PP(W)}))$ respectively. Substituting $c_1(T_{\PP(W)}) = 4h$, 
$c_2(T_{\PP(W)}) = 6h^2$ and $c_3(T_{\PP(W)}) = 4h^3$ we conclude that the discriminant is an octic surface and its singular locus has class
$4(4h\cdot 6h^2 - 4h^3) = 4(24h^3 - 4h^3) = 80h^3$.
For generic $X_{b4}$ this shows that the corank 2 locus consists of 80 points, and these points are the only singularities of the discriminant.
\end{proof}

Using this description one can decompose the derived category.

\begin{proposition}
Let $X = X_{b4}$. Then there is a semiorthogonal decomposition 
\begin{equation*}
\BD(X_{b4}) = \langle \CO_X, \CO_X(h), \CO_X(2h), \CO_X(3h), \BD(\PP(W),\Cliff_0) \rangle,
\end{equation*}
where 
\begin{equation*}
\Cliff_0 = \CO_{\PP(W)} \oplus \Omega^2_{\PP(W)}
\end{equation*}
is the sheaf of even parts of Clifford algebras associated with the conic bundle $X \to \PP(W)$.
The category $\BD(\PP(W),\Cliff_0)$ is a twisted $3$-Calabi--Yau category.
\end{proposition}
\begin{proof}
By~\cite[Thm.~4.2]{kuznetsov2008derived} a conic bundle structure of $X$ gives a two-component semiorthogonal decomposition $\BD(X) = \langle \BD(\PP(W)), \BD(\PP(W),\Cliff_0) \rangle$.
Taking into account the standard exceptional collection in the derived category of the projective space $\PP(W)$ we deduce 
the first claim. The explicit formula for the even part of the Clifford algebra in this case is~\cite[(12)]{kuznetsov2008derived}.
Finally, to check the Calabi--Yau property of $\BD(\PP(W),\Cliff_0)$ we use Theorem~5.3 of~\cite{kuznetsov2014semiorthogonal}.
For this we note that $\Fl(1,3;W)$ is a $\PP^2$-bundle over $\PP(W)$ and so it has a rectangular Lefschetz decomposition
\begin{equation*}
\BD(\Fl(1,3;W)) = \langle \BD(\PP(W)), \BD(\PP(W)) \otimes \CO(1,1), \BD(\PP(W))\otimes \CO(2,2) \rangle.
\end{equation*}
Applying the Theorem we deduce that the square of the Serre functor of the category $\BD(\PP(W),\Cliff_0)$ 
is isomorphic to the shift by 6. Hence the composition of the Serre functor with the shift by $-3$ is an involution $\tau$,
so the category can be considered as a $\tau$-twisted 3 Calabi--Yau category.
\end{proof}

\begin{remark}
One can consider the equivariant with respect to the involution $\tau$ category $\BD(\PP(W),\Cliff_0)^\tau$ 
(in the spirit of~\cite{kuznetsov2014derived}). This category is an untwisted 3 Calabi--Yau category, 
it would be interesting to investigate it.
\end{remark}

\subsection{Type (b9)}

In fact, these varieties were recently thoroughly investigated by Cinzia Casagrande in~\cite{casagrande2014rank}. In particular, she proved the following

\begin{theorem}[{\protect{\cite[Sec.~3.2]{casagrande2014rank}}}]
Let $X_{b9}$ be a K\"uchle fourfold of type {\rm(b9)}. Then there are $7$ points $x_1,\dots,x_7 \in \PP^4$ in a general position
such that $X_{b9}$ is obtained from the blowup $\Bl_{x_1,\dots,x_7}\PP^4$ in the $7$ points by $22$ antiflips in the proper preimages
of $21$ lines $L_{ij}$ passing through $x_i$ and $x_j$ and the proper preimage of the rational normal quartic $C$ passing through
all seven points.
\end{theorem}

This result has an immediate consequence for the structure of the derived category.

\begin{corollary}
The derived category $\BD(X_{b9})$ has a full exceptional collection of length $48$.
\end{corollary}
\begin{proof}
The projective space $\PP^4$ has an exceptional collection of length 5.
When we blowup 7 points, by Orlov's blowup formula, we should add 3 exceptional
objects for each of the points, which gives $7\cdot 3 = 21$ more exceptional objects.
Finally, for each of 22 anti-flips by~\cite{bondal1995semiorthogonal} we should add one more object.
So, in the end we get an exceptional collection of length $5 + 21 + 22 = 48$.
\end{proof}

\section{Varieties of type (d3)}\label{section-d3}

By definition the K\"uchle 4-fold of type (d3) is the zero locus of a generic section 
of the vector bundle $\Lambda^2\CU^\vee \oplus \Lambda^2\CU^\vee \oplus \CO(1)$ on $\Gr(5,10)$. 
Note that it is a hyperplane section of a fivefold $M$ defined as the zero locus
of a section of $\Lambda^2\CU^\vee \oplus \Lambda^2\CU^\vee$ on $\Gr(5,10)$. We start with
its description. In fact, we prove a more general result below by replacing $\Gr(5,10)$
by $\Gr(n,2n)$ with arbitrary $n$.

\subsection{Lagrangian spaces for a pencil of skew-forms}

Let $V$ be a vector space of dimension $2n$ and let $\lambda:\kk^2 \to \Lambda^2V^\vee$ be a pencil
of skew-forms. Let $M_\lambda$ be the zero locus of the global section of the vector bundle $\Lambda^2\CU^\vee \oplus \Lambda^2\CU^\vee$
on $\Gr(n,V)$ corresponding to this pencil. Clearly, it parameterizes all half-dimensional subspaces
in $V$ which are Lagrangian for all skew-forms in the pencil.

\begin{theorem}\label{ttf}
If $M_\lambda$ is smooth then $M_\lambda \cong (\PP^1)^n$.
\end{theorem}
\begin{proof}
Consider the line $L_\lambda \subset \PP(\Lambda^2V^\vee)$ corresponding to the pencil $\lambda$ and
the discriminant hypersurface $D \subset \PP(\Lambda^2V^\vee)$ corresponding
to degenerate skew-forms. Since the degeneracy condition for a skew-form is equivalent
to the vanishing of its determinant, which in its turn equals to the square of the Pfaffian,
we see that $D$ is a hypersurface of degree $n$. Consequently, the intersection $L_\lambda \cap D$
consists of $n$ points (counted with appropriate multiplicity) unless the line is contained in $D$.
For each intersection point $\lambda_i \in L_\lambda \cap D$ denote by $K_i \subset V$
the kernel of the skew-form $\lambda_i$. Let us show that the smoothness of $M_\lambda$ is equivalent 
to the following two conditions:
\begin{enumerate}
\item the intersection $L_\lambda \cap D = \{ \lambda_1,\lambda_2,\dots,\lambda_n \}$ consists of $n$ distinct points, and
\item for each $1 \le i \le n$ we have $\dim K_i = 2$ and $V = K_1 \oplus \dots \oplus K_n$ is a direct sum decomposition.
\end{enumerate}
To prove this, note that given a skew-symmetric form $\lambda_0$ the zero locus of the corresponding section of~$\Lambda^2\CU^\vee$ 
on $\Gr(n,V)$ is singular at a point corresponding to an $n$-dimensional subspace $U \subset V$
if and only if $\dim(U \cap \Ker\lambda_0) \ge 2$ --- this can be checked by a simple local computation.

Now assume first that $\dim K_i > 2$ for some $i$. Choose a skew-form $\lambda_0$ in the pencil $\lambda$
distinct from~$\lambda_i$. Let $K_0 \subset K_i$ be a 2-dimensional subspace which is isotropic for $\lambda_0$
(it exists since $\dim K_i > 2$). Let $K_0^\perp \subset V$ be the orthogonal of $K_0$ with respect to $\lambda_0$.
Then $K_0 \subset K_0^\perp$ since $K_0$ is $\lambda_0$-isotropic. Let $V' := K_0^\perp/K_0$, $\dim V' = 2n - 4$.
Note that $K_0$ is in the kernel of the restriction to $K_0^\perp$ of both $\lambda_i$ and~$\lambda_0$, hence
both forms induce skew-forms $\lambda'_i$ and $\lambda'_0$ on $V'$. Let $U' \subset V'$ be a subspace of dimension $n-2$
isotropic for both $\lambda'_i$ and $\lambda'_0$ (it exists by a dimension count). Let $U \subset K_0^\perp \subset V$
be the preimage of $U'$ under the projection $K_0^\perp \to K_0^\perp/K_0 = V'$. 
By construction $U$ is an $n$-dimensional subspace in $V$ isotropic for both forms $\lambda_i$ and $\lambda_0$.
Moreover, $\dim(U \cap K_i) \ge 2$, hence the zero locus of $\lambda_i$ on $\Gr(n,V)$ is singular at $U$, and on the other
hand $U$ is also contained in the zero locus of $\lambda_0$. Hence it gives a singular point in the intersection 
of the zero loci, i.e.\ on $M_\lambda$. This shows that for smooth $M_\lambda$ we have $\dim K_i = 2$ for all $i$.

Note also that the above argument also proves that the kernel $K_i$ of $\lambda_i$ is not isotropic for the other
skew-forms in the pencil (otherwise we could take $K_0 = K_i$ and repeat the argument) and
this implies the transversality of $L_\lambda$ and $D$. Indeed, the tangent space to $D$ at a point $\lambda_i$ 
is the space of all skew-forms vanishing on $K_i$, and as we just noted the line $L_\lambda$ is not contained in it. 
The transversality shows that there are precisely $n$ distinct points of intersection of $L_\lambda$ with $D$. 

So, it remains to show the direct sum decomposition. Take a nondegenerate skew-form $\lambda_0$ in the pencil. Let us 
show that all the kernel spaces $K_i$ are mutually orthogonal with respect to $\lambda_0$, i.e.\ that $\lambda_0(K_i,K_j) = 0$
for all $i \ne j$. For this note that $\lambda_0$ is a linear combination of $\lambda_i$ and $\lambda_j$, and so it is enough
to show that $\lambda_i(K_i,K_j) = 0$ and $\lambda_j(K_i,K_j) = 0$. But both equalities are evident since $K_i$ is the kernel 
of $\lambda_i$ and $K_j$ is the kernel of $\lambda_j$. As also the restriction of $\lambda_0$ to each $K_i$ is nondegenerate,
the direct sum decomposition follows.

\begin{remark}
In fact what we proved is that $M_\lambda$ is smooth if and only if the pencil of skew-forms in appropriate basis can be written 
in a block-diagonal form
\begin{equation*}
\lambda = u \mathop{\mathsf{diag}}\left( \left(\begin{smallmatrix} 0 & 1 \\ -1 & 0 \end{smallmatrix}\right), 
\left(\begin{smallmatrix} 0 & 1 \\ -1 & 0 \end{smallmatrix}\right), \dots,
\left(\begin{smallmatrix} 0 & 1 \\ -1 & 0 \end{smallmatrix}\right) \right) 
+ v \mathop{\mathsf{diag}}\left( \left(\begin{smallmatrix} 0 & -a_1 \\ a_1 & 0 \end{smallmatrix}\right), 
\left(\begin{smallmatrix} 0 & -a_2 \\ a_2 & 0 \end{smallmatrix}\right), \dots,
\left(\begin{smallmatrix} 0 & -a_n \\ a_n & 0 \end{smallmatrix}\right) \right),
\end{equation*}
where $u$ and $v$ are appropriate coordinates on the pencil and $a_1,\dots,a_n \in \kk$ are pairwise distinct. This is an analogue
of a standard form for a pencil of quadrics.
\end{remark}

Now let us show that $M_\lambda$ is isomorphic to the product $\PP(K_1)\times\PP(K_2) \times \dots \times \PP(K_n) \cong (\PP^1)^n$.
For this we define a map $M_\lambda \to \prod \PP(K_i)$ as follows. Let $U \subset V$ be an $n$-dimensional vector subspace 
isotropic for the pencil $\lambda$. The pencil of skew-forms $\lambda$ gives a pencil of maps $\lambda:U \to U^\perp$. 
The Pfaffian of a skew-form $\lambda$ in the pencil equals the determinant of the corresponding 
map $U \to U^\perp$, hence the maps $\lambda_1,\dots,\lambda_n$ are degenerate. This means that $U$ intersects nontrivially
all the kernel spaces $K_i$. But then it follows that
\begin{equation*}
U = (U\cap K_1) \oplus (U \cap K_2) \oplus \dots \oplus (U \cap K_n),
\qquad
\text{and}
\qquad
\text{$\dim(U\cap K_i) = 1$ for each $i$.}
\end{equation*}
Indeed, each summand in the right-hand side is at least 1-dimensional, and the sum is a direct sum, hence the dimension
of the right-hand side is at least $n$. On the other hand, the right-hand side is evidently contained in the left-hand side,
and the dimension of that is also $n$. Hence the sides coincide and the dimension of each summand in the right-hand side is $1$.
This means that 
\begin{equation*}
M \to \PP(K_1)\times\PP(K_2)\times\dots\times\PP(K_n),
\qquad
U \mapsto (U\cap K_1,U \cap K_2,\dots,U\cap K_n)
\end{equation*}
is a well defined map. The inverse map, of course, is given by 
\begin{equation*}
\PP(K_1)\times\PP(K_2)\times\dots\times\PP(K_n) \to M,
\qquad
(u_1,u_2,\dots,u_n) \mapsto \kk u_1 \oplus \kk u_2 \oplus \dots \oplus \kk u_n.
\end{equation*}
To check that it is well defined we should show that $\lambda(u_i,u_j) = 0$ for any skew-form $\lambda$ in the pencil 
and all $u_i \in K_i$, $u_j \in K_j$. But this follows from the fact that $\lambda$ is a linear combination of $\lambda_i$
and $\lambda_j$ and $\lambda_i(u_i,u_j) = \lambda_j(u_i,u_j) = 0$ since $u_i \in K_i = \Ker\lambda_i$ and $u_j \in K_j = \Ker\lambda_j$
(actually, we already used this argument before).

The constructed two maps are clearly mutually inverse, so this proves the Theorem.
\end{proof}

\subsection{A hyperplane section}

In the previous paragraph we proved that a smooth zero locus of a section of the vector bundle $\Lambda^2\CU^\vee \oplus \Lambda^2\CU^\vee$ 
on $\Gr(n,2n)$ is isomorphic to $(\PP^1)^n$. Now we can describe its hyperplane section.

\begin{lemma}
Under the isomorphism of Theorem~$\ref{ttf}$ the line bundle $\CO_{\Gr(n,2n)}(1)_{|M_\lambda}$ identifies with the line bundle $\CO(1,1,\dots,1)$
on $(\PP^1)^n$. In particular, the zero locus of a general section of the vector bundle $\Lambda^2\CU^\vee \oplus \Lambda^2\CU^\vee \oplus \CO(1)$ 
on $\Gr(n,2n)$ is isomorphic to a divisor of multidegree $(1,1,\dots,1)$ in $(\PP^1)^n$. 
\end{lemma}
\begin{proof}
The precise form of the isomorphism of Theorem~\ref{ttf} shows that the restriction of the tautological bundle $\CU$ from $\Gr(n,2n)$
to $(\PP^1)^n$ splits as
\begin{equation*}
\CU_{|(\PP^1)^n} = \bigoplus_{i=1}^n \, p_i^*\CO(-1),
\end{equation*}
where $p_i$ is the projection of $(\PP^1)^n$ to the $i$-th factor. As $\CO_{\Gr(n,2n)}(-1) = \det \CU$, it follows that its 
restriction to $(\PP^1)^n$ is isomorphic to $\otimes_{i=1}^n p_i^*\CO(-1) = \CO(-1,-1,\dots,-1)$, which proves the first statement.
The second statement is evident.
\end{proof}

Consider a divisor of multidegree $(1,1,\dots,1)$ on $\PP(K_1)\times\PP(K_2)\times\dots\PP(K_n)$, where $K_1,K_2,\dots,K_n$ are 2-dimensional
vector spaces. Such a divisor is given by a multilinear form $s \in K_1^\vee \otimes K_2^\vee \otimes \dots \otimes K_n^\vee$,
which can be also thought of as a linear map 
\begin{equation*}
s:K_n \to K_1^\vee \otimes K_2^\vee \otimes \dots \otimes K_{n-1}^\vee.
\end{equation*}
The latter gives a pencil of sections of the line bundle $\CO(1,1,\dots,1)$ on $\PP(K_1)\times\PP(K_2)\times\dots\PP(K_{n-1})$.
Denote by $Z \subset \PP(K_1)\times\PP(K_2)\times\dots\PP(K_{n-1})$ the intersection of the corresponding divisors.

\begin{lemma}
The projection $\PP(K_1)\times\PP(K_2)\times\dots\PP(K_n) \to \PP(K_1)\times\PP(K_2)\times\dots\PP(K_{n-1})$ identifies
the zero locus $X$ of the section $s$ of $\CO(1,1,\dots,1)$ with the blowup of $\PP(K_1)\times\PP(K_2)\times\dots\PP(K_{n-1})$ in $Z$.
In particular, $X$ is rational.
\end{lemma}
\begin{proof}
Evident. 
\end{proof}

Considering the case $n = 5$ we obtain a description of K\"uchle 4-folds of type (d3).

\begin{corollary}
A K\"uhle fourfold $X_{d3}$ is isomorphic to the blowup of $(\PP^1)^4$ in a {\rm K3}-surface $Z$
defined as the intersection of two divisors of multidegree $(1,1,1,1)$.
\end{corollary}
\begin{proof}
The only thing to check is that $Z$ defined as above is a K3 surface. But this follows immediately
from the adjunction formula and Lefschetz Theorem. 
\end{proof}

Let $h_i$ denote the pullback to $X_{d3}$ of the ample generator of the Picard groups of $\PP(K_i)$.

\begin{corollary}
There is a semiorthogonal decomposition
\begin{equation*}
\BD(X_{d3}) = \left\langle \left\{ \CO_X\left(k_1h_1 + k_2h_2 + k_3h_3 + k_4h_4 \right) \right\}_{0 \le k_i \le 1}, \BD(Z) \right\rangle.
\end{equation*}
consisting of $16$ exceptional line bundles and the derived category of a {\rm K3}-surface $Z$.
\end{corollary}
\begin{proof}
This is an immediate consequence of the Orlov's blowup formula applied to the standard exceptional collection
on the product of projective lines.
\end{proof}

\begin{remark}
One could also predict a K3 category in $\BD(X_{d3})$ with the help of \cite[Theorem~5.3]{kuznetsov2014semiorthogonal}.
Indeed, $\BD(M)$ clearly has a rectangular Lefschetz decomposition of length $2$, so applying the Theorem,
one immediately gets such a conclusion.
\end{remark}

\section{Varieties of type (c7)}\label{section-c7}

By definition the K\"uchle 4-fold of type (c7) is the zero locus of a generic section 
of the vector bundle $\Lambda^2\CU^\perp(1) \oplus \CO(1)$ on $\Gr(3,8)$. 
Again, as in the case with type (d3) we first consider a 5-fold $M$ defined as the zero locus
of a section of $\Lambda^2\CU^\perp(1)$. For convenience, we replace the Grassmanian $\Gr(3,8)$
with $\Gr(5,8)$. As the bundle $\CU^\perp$ on the first is isomorphic to the bundle $\CU$ on the second,
the 5-fold $M$ is the zero locus of a generic section of the bundle $\Lambda^2\CU(1) \cong \Lambda^3\CU^\vee$ on $\Gr(5,8)$.

Let $\lambda \in \Lambda^3V^\vee$ be a generic 3-form on a vector space $V$ of dimension $8$.
Let $M \subset \Gr(5,V)$ be the zero locus of $\lambda$ considered as a section of the vector bundle~$\Lambda^3\CU^\vee$.
In other words, $M$ is the locus of all 5-subspaces $U \subset V$ such that the restriction of $\lambda$ to~$U$ is zero
(one could call such $U$ isotropic). We will show below that $M$ is isomorphic to the blowup of $\PP^5$
in the Veronese surface.

\subsection{The blowup of $\PP^5$ in the Veronese surface}

Let $W$ be a vector space of dimension $3$. Consider the Veronese embedding $\PP(W) \to \PP(S^2W)$
and let $Y$ be the blowup of $\PP(S^2W)$ with the center in $\PP(W)$.

\begin{lemma}\label{lemma-blowup-veronese}
There is an embedding $Y \subset \PP(S^2W) \times \PP(S^2W^\vee)$ such that
\begin{equation*}
Y = \{ (C,C')\in \PP(S^2W)\times \PP(S^2W^\vee)\ |\ C\cdot C' = t \cdot 1_W \text{ for some $t \in \kk$}\},
\end{equation*}
where $C\cdot C' \in \PP(W\otimes W^\vee)$ is the image of $C\otimes C'$ under the natural map $S^2W\otimes S^2W^\vee \to W\otimes W^\vee$.
The projections $\pi:Y \to \PP(S^2W)$ and $\pi':Y \to \PP(S^2W^\vee)$ are the blowups of the Veronese surfaces
$\PP(W) \subset \PP(S^2W)$ and $\PP(W^\vee) \subset \PP(S^2W^\vee)$. The Picard group of $Y$ is generated by
the pullbacks $H$ and $H'$ of the positive generators of $\Pic(\PP(S^2W))$ and $\Pic(\PP(S^2W^\vee))$ respectively.
The exceptional divisors $E$ and $E'$ of the blowups $\pi$ and $\pi'$ are expressed as
\begin{equation*}
E = 2H - H',
\qquad
E' = 2H' - H
\end{equation*}
and the canonical class is 
\begin{equation*}
K_Y = 2E - 6H = 2E' - 6H' = -2(H + H').
\end{equation*}
\end{lemma}
\begin{proof}
Consider the rational map $\PP(S^2W) \to \PP(S^2W^\vee)$ defined by $C \mapsto \widehat{C}$, where $\widehat{C}$ is the adjoint matrix of $C$. One can check that $Y$ is its graph.
Alternatively, $Y$ can be identified with the space of {\sf complete quadrics} in $\PP(W)$. All the properties of $Y$
are well known.
\end{proof}

Consider the natural $\GL(W)$-action on $Y$.

\begin{lemma}
The action of $\GL(W)$ on $Y$ has precisely $4$ orbits:
\begin{equation*}
\begin{array}{ll}
Y_0 = \{(C,C')\ |\ r(C) = 3,\ r(C') = 3,\ C' = \widehat{C} \},
\qquad&
y_0 = 
\left( \left( \begin{smallmatrix} 1 & 0 & 0 \\ 0 & 1 & 0 \\ 0 & 0 & 1 \end{smallmatrix} \right), 
\left( \begin{smallmatrix} 1 & 0 & 0 \\ 0 & 1 & 0 \\ 0 & 0 & 1 \end{smallmatrix} \right) \right),\\[2ex]
Y_1 = \{(C,C')\ |\ r(C) = 2,\ r(C') = 1,\ C' = \widehat{C} \},
\qquad&
y_1 = 
\left( \left( \begin{smallmatrix} 1 & 0 & 0 \\ 0 & 1 & 0 \\ 0 & 0 & 0 \end{smallmatrix} \right), 
\left( \begin{smallmatrix} 0 & 0 & 0 \\ 0 & 0 & 0 \\ 0 & 0 & 1 \end{smallmatrix} \right) \right),\\[2ex]
Y_2 = \{(C,C')\ |\ r(C) = 1,\ r(C') = 2,\ C = \hat{C'} \},
\qquad&
y_2 = 
\left( \left( \begin{smallmatrix} 1 & 0 & 0 \\ 0 & 0 & 0 \\ 0 & 0 & 0 \end{smallmatrix} \right), 
\left( \begin{smallmatrix} 0 & 0 & 0 \\ 0 & 1 & 0 \\ 0 & 0 & 1 \end{smallmatrix} \right) \right),\\[2ex]
Y_3 = \{(C,C')\ |\ r(C) = 1,\ r(C') = 1,\ C\cdot C' = 0 \},
\qquad&
y_3 = 
\left( \left( \begin{smallmatrix} 1 & 0 & 0 \\ 0 & 0 & 0 \\ 0 & 0 & 0 \end{smallmatrix} \right), 
\left( \begin{smallmatrix} 0 & 0 & 0 \\ 0 & 0 & 0 \\ 0 & 0 & 1 \end{smallmatrix} \right) \right),
\end{array}
\end{equation*}
the points $y_0,y_1,y_2,y_3 \in Y$ being typical repreentatives of the orbits.
\end{lemma}
\begin{proof}
A straightforward computation.
\end{proof}

Consider the maps 
\begin{equation*}
\begin{array}{l}
c:\Lambda^2W^\vee\otimes\CO_Y(-H) \to W\otimes W^\vee\otimes\CO_Y,
\qquad
\xi' \mapsto C \cdot \xi',\\
c':\Lambda^2W\otimes\CO_Y(-H') \to W\otimes W^\vee\otimes\CO_Y,
\qquad
\xi \mapsto \xi \cdot C'.
\end{array}
\end{equation*}

\begin{lemma}\label{lemma-fg}
The sum of the maps $c$ and $c'$
\begin{equation*}
\Lambda^2W^\vee\otimes\CO_Y(-H) \oplus \Lambda^2W\otimes\CO_Y(-H') \xrightarrow{\ c + c'\ } W\otimes W^\vee\otimes\CO_Y,
\end{equation*}
has contant rank $3$. Its image $\fg := \Im(c + c')$ is a vector subbundle of rank $3$ in $\fsl(W)\otimes\CO_X$,
each fiber of which is a Lie subalgebra.
\end{lemma}
\begin{proof}
The map is evidently $\GL(W)$-equivariant, so it suffices to compute it at points $y_0,y_1,y_2,y_3$.
A direct computation shows that the images at these points are the spaces of $3\times 3$ matrices 
of the following form
\begin{equation*}
\fg_0 = \left \{ \left( \begin{smallmatrix} 0 & a & b \\ -a & 0 & c \\ -b & -c & 0 \end{smallmatrix} \right) \right \},
\quad
\fg_1 = \left \{ \left( \begin{smallmatrix} 0 & a & b \\ -a & 0 & c \\ 0 & 0 & 0 \end{smallmatrix} \right) \right \},
\quad
\fg_2 = \left \{ \left( \begin{smallmatrix} 0 & a & b \\ 0 & 0 & c \\ 0 & -c & 0 \end{smallmatrix} \right) \right \},
\quad
\fg_3 = \left \{ \left( \begin{smallmatrix} 0 & a & b \\ 0 & 0 & c \\ 0 & 0 & 0 \end{smallmatrix} \right) \right \},
\end{equation*}
respectively. Each space is 3-dimensional and consists of matrices with zero trace, hence the first claim. 
Moreover, each subspace is a Lie subalgebra in $\fsl(W)$, hence the second claim.
\end{proof}

\begin{remark}\label{fso}
Note that $\fg_0 = \fso(W,C')$ is the special orthogonal (with respect to the form $C'$) Lie algebra.
\end{remark}

\begin{remark}\label{funny}
Note also that each of the subalgebras has the following funny property --- for each vector in $W$
and for each covector in $W^\vee$ there is a nonzero element $\xi$ in the subalgebra annihilating this (co)vector.
\end{remark}

\subsection{Isotropic 5-spaces for a 3-form on an 8-space}

Now we can relate the variety $Y$ discussed in the previous subsection to the scheme $M \subset \Gr(5,V)$
of isotropic 5-spaces in a vector space $V$ of dimension~8 for a 3-form $\lambda \in \Lambda^3V^\vee$.
Indeed, we can identify the space $V$ with the Lie algebra $\fsl(W)$. Since $\fsl(W)$ has a natural
nondegenerate scalar product (the Killing form), we can consider the orthogonal $\fg^\perp$
of the subbundle $\fg$ defined in Lemma~\ref{lemma-fg} with respect to it. Then $\fg^\perp \subset \fsl(W)\otimes\CO_Y = V\otimes\CO_Y$
is a vector subbundle of rank 5, and so it gives a morphism 
\begin{equation*}\label{equation-phi}
\varphi:Y \to \Gr(5,V).
\end{equation*} 
It remains to construct a 3-form on $V$ and to show that $Y$ is identified by this map with $M \subset \Gr(5,V)$.


\begin{lemma}\label{lemma-lambda}
There is a unique $\SL(W)$-invariant $3$-form $\lambda$ on the adjoint representation $\fsl(W)$ of $\SL(W)$. 
It is given by the formula 
\begin{equation*}\label{l38}
\lambda(\xi_1,\xi_2,\xi_3) = \Tr([\xi_1,\xi_2]\xi_3).
\end{equation*} 
\end{lemma}
\begin{proof}
Let us show that the right-hand-side is a skew-form. Indeed, $\lambda$ is clearly skew-symmetric in the first two arguments,
so it remains to show that it is invariant under the cyclic permutation. But this follows immediately from
the invariance of the trace under the cyclic permutation.
%
Finally, a direct computation with the Littlewood--Richardson rule gives
\begin{equation*}
\Lambda^3(\fsl(W)) \cong  \Sigma^{2,0,-2}W \oplus S^3W^\vee \oplus S^3W \oplus \fsl(W) \oplus \kk,
\end{equation*}
an isomorphism of $\SL(W)$-modules. We see that the trivial module has multiplicity 1,
hence an invariant 3-form is unique.
\end{proof}

\begin{proposition}\label{laku}
The form~$\lambda$ is the generic $3$-form on a vector space of dimension $8$.
\end{proposition}
\begin{proof}
Let $V$ be a vector space of dimension 8. Recall that by~\cite[Table 1]{elashvili1972canonical} 
the action of the group $\SL(V)$ on the space $\PP(\Lambda^3V^\vee)$ of 3-forms has a dense orbit, 
the Lie algebra of the stabilizer of the generic 3-form is $\fsl_3 \subset \fsl(V)$, 
and the restriction of the representation $V$ to this $\fsl_3$ is isomorphic 
to the adjoint representation. 
Thus we can identify 
\begin{equation*}
V \cong \fsl(W)
\end{equation*}
so the generic 3-form on $V$ identifies with an $\SL(W)$-invariant 3-form on $\fsl(W)$.
As we have shown that $\lambda$ is the unique such form, the Proposition follows.
\end{proof}

\begin{remark}
In an appropriate basis the form $\lambda$ can be written explicitly as
\begin{equation*}
\lambda = x_{238} + x_{167} - x_{247} - x_{356} - x_{148} - x_{158}.
\end{equation*}
\end{remark}

%

\begin{proposition}
The map $\varphi:Y \to \Gr(5,V)$ is a closed embedding and its image coincides with the zero locus $M$
of the global section $\lambda$ of the vector bundle $\Lambda^3\CU^\vee$.
\end{proposition}
\begin{proof}
First let us show that the image of the map $\varphi$ is contained in $M$, i.e.\ consists of 5-dimensional subspaces
isotropic for $\lambda$. Since $Y$ is irreducible and the condition of being isotropic is closed,
it is enough to check that the image of the open $\GL(W)$-orbit $Y_0 \subset Y$ consists 
of isotropic subspaces. Moreover, since the map $\varphi$ and the form $\lambda$ are $\GL(W)$-invariant,
it is enough to check only that the image of the point $y_0$ is isotropic. This can be done
by a straightforward computation. Indeed, we have $\varphi(y_0) = \fg_0^\perp$ is the space
of all symmetric $3\times 3$ matrices with zero trace. Take any $\xi_1,\xi_2,\xi_3 \in \fg_0^\perp$.
Then it is clear that $[\xi_1,\xi_2]$ is a skew-symmetric matrix, so it is in $\fg_0$, hence
$\Tr([\xi_1,\xi_2]\xi_3) = 0$.

Second, let us check that the map $\varphi$ is an embedding on the open orbit $Y_0$.
Indeed, this is clear since a nondegenerate quadratic form $C'$ can be reconstructed
from the Lie algebra $\fg = \fso(W,C') \subset \fsl(W)$ as the unique $\fg$-invariant
quadratic form in $S^2W^\vee$.

Thus we just checked that the morphism $\varphi$ is a morphism $Y \to M \subset \Gr(5,V)$ and this morphism is injective
on the open subset $Y_0 \subset Y$. Since both $Y$ and $M$ are 5-dimensional,
the morphism is birational. Since both are smooth and have Picard number 2 (for $Y$ this is clear by definition
and for $M$ this is proved in~\cite{kuchle1995fano}), this is an isomorphism.
\end{proof}

Thus we have proved the following

\begin{theorem}\label{kc7}
The zero locus $M$ of a generic section of the vector bundle $\Lambda^3\CU^\vee$ on $\Gr(5,8)$
is isomorphic to the blowup of $\PP^5$ in Veronese surface $\nu_2(\PP^2)$.
\end{theorem}

\subsection{A hyperplane section}

Recall that by definition the K\"uchle fourfold of type (c7) is a half-anticanonical section of the variety $M$.
By Lemma~\ref{lemma-blowup-veronese} the corresponding linear system on $M$ is $3H - E$, where $H$ is the pullback
to $M$ of the generator of $\Pic(\PP^5)$, and $E$ is the exceptional divisor of the blowup $M \to \PP^5$.

\begin{corollary}
Let $X$ be a K\"uchle fourfold of type $(c7)$. Then there is a cubic fourfold $Z$ containing a Veronese surface
$S = \nu_2(\PP^2)$ such that $X$ is isomorphic to the blowup of $Z$ in $S$.
\end{corollary}
\begin{proof}
Evident.
\end{proof}

Now we can also describe the derived category. Denote by $E_X$ the exceptional divisor of the blowup map 
$X \to Z$ and by $i:E_X \to X$ its embedding. Note that $E_X$ is a $\PP^1$-bundle over $S$.
Denote by $h$ the positive generator of $\Pic(S)$ and its pullback to $E_X$.

\begin{corollary}
There is a semiorthogonal decomposition
\begin{equation*}
\BD(X_{c7}) = \langle \CO_X, \CO_X(H), \CO_X(2H), i_*\CO_{E_X}, i_*\CO_{E_X}(h), i_*\CO_{E_X}(2h), \CA_X \rangle,
\end{equation*}
where $\CA_X$ is a noncommutative {\rm K3} cateogry, equivalent to the nontrivial part of the derived category
of the cubic fourfold $Z$.
\end{corollary}
\begin{proof}
The derived category of $Z$ has a decomposition $\BD(Z) = \langle \CO_Z, \CO_Z(H), \CO_Z(2H), \CA_Z \rangle$
with $\CA_Z$ being a noncommutative K3 category (see~\cite{kuznetsov2010derived}). Since $X$ is the blowup
of $Z$ in $S$ we have by Orlov's blowup formula $\BD(X) = \langle \CO_X, \CO_X(H), \CO_X(2H), \CA_Z, \BD(S) \rangle$.
Mutating $\CA_Z$ to the right and replacing $\BD(S) = \BD(\PP^2)$ by the standard exceptional collection
we get the result.
\end{proof}

\begin{remark}
Again, one could predict appearance of a K3 category by using \cite[Theorem~5.3]{kuznetsov2014semiorthogonal}.
Indeed, the blowup $M$ of $\PP^5$ in the Veronese surface has an exceptional collection consisting of
$6$ line bundles pulled back from $\PP^5$ and $6$ line bundles sitting on the exceptional divisor.
Mutating, one can arrange them into a rectangular Lefschetz decomposition consisting of two blocks and then
apply the Theorem.
\end{remark}


\end{document}